\documentclass[reqno,oneside,11pt]{amsart}
\usepackage{amssymb,enumerate,bm}
\usepackage[T1]{fontenc}
\usepackage[width=14cm,centering]{geometry}
\usepackage[colorlinks=true, allcolors=blue]{hyperref}
\usepackage{tcolorbox, stmaryrd, mathrsfs}
\usepackage{color}
\usepackage{parskip}
\usepackage{mathtools}
\usepackage[shortlabels]{enumitem}
\newtheorem{theorem}{Theorem}[section]

\newtheorem{prop}[theorem]{Proposition}

\theoremstyle{definition}

\theoremstyle{remark}

\newtheorem{definition}[theorem]{Definition}

\newcommand{\MT}{\text{MT}}
\def\R{\mathbb{R}}

\newcommand\restr[2]{{%
  \left.\kern-\nulldelimiterspace %
  #1 %
  \vphantom{\big|} %
  \right|_{#2} %
  }}

\definecolor{darkorchid}{rgb}{0.6,0.196,0.8} %
\newcommand{\robin}[1]{{\color{darkorchid}[[\textbf{Robin: }#1]]}} 
\definecolor{darkblue}{rgb}{0, 0, 1} %

\usepackage[normalem]{ulem}

\definecolor{darkgrn}{rgb}{0, 0.75, 0}

\definecolor{eaclr}{rgb}{0.2, 0.7, 0}

\newcommand{\KJ}[1]{{\color{blue}[[\textbf{KJ: }#1]]}}

\title[Any Graph is a Mapper Graph]
    {Any Graph is a Mapper Graph}

    \author[Enrique G Alvarado]{Enrique G Alvarado}
    \author[Robin Belton]{Robin Belton}
    \author[Kang-Ju Lee]{Kang-Ju Lee}
    \author[Sourabh Palande]{Sourabh Palande}
    \author[Sarah Percival]{Sarah Percival}
    \author[Emilie Purvine]{Emilie Purvine}
    \author[Sarah Tymochko]{Sarah Tymochko}

    \address{Department of Mathematics, University of California at Davis, Davis, CA, USA}
    \email{ealvarado@math.ucdavis.edu}
    \address{Department Mathematical Sciences, Smith College, Northampton, MA, USA}
    \email{rbelton@smith.edu}
    \address{Research Institute of Mathematics, Seoul National University, Seoul, South Korea}
    \email{leekj0706@snu.ac.kr}
    \address{Department of Computational Mathematics, Science, and Engineering, Michigan State University, East Lansing, MI, USA}
    \email{palandes@msu.edu}
    \address{Department of Mathematics and Statistics, University of New Mexico, Albuquerque, NM, USA}
    \email{spercival@unm.edu}
    \address{Mathematics, Statistics, and Data Science, Pacific Northwest National Laboratory, Seattle, WA, USA}
    \email{Emilie.Purvine@pnnl.gov}
    \address{Department of Mathematics, University of California, Los Angeles, Los Angeles, CA, USA}
    \email{tymochko@math.ucla.edu}

\keywords{Mapper, Adaptability, Topological data analysis, Nerve}
\subjclass[2020]{62R40}
\thanks{KL was supported in part by the National Research Foundation of Korea (NRF) Grants funded by the Korean
Government (MSIP) (No.2021R1C1C2014185).}
\thanks{SP was supported in part by the NSF through grant CCF-2142713.}

\begin{document}

\begin{abstract}
The Mapper algorithm is a popular tool for visualization and data exploration in topological data analysis. We investigate an inverse problem for the Mapper algorithm: Given a dataset $X$ and a graph $G$, does there exist a set of Mapper parameters such that the output Mapper graph of $X$ is isomorphic to $G$? We provide constructions that affirmatively answer this question. Our results demonstrate that it is possible to engineer Mapper parameters to generate a desired graph. %
\end{abstract}

\maketitle 

\section{Introduction}
\label{sec:intro}

Topological data analysis (TDA) uses methods from topology to study the underlying structure and ``shape'' of a dataset. We refer the reader to \cite{carlsson2009topology} for an overview of this area. This paper focuses on \emph{Mapper}, introduced in \cite{singh2007topological}, a popular visualization tool in TDA that constructs a network representation of a dataset. Mapper has been used in many applications that include analyzing breast cancer data \cite{nicolau2011topology}, identifying diabetes subgroups \cite{li2015identification}, and analyzing Passiflora leaves \cite{percival2024topological}.

To apply the Mapper algorithm, the user needs to determine the following parameters: a \emph{lens} (or \emph{filter}) function $f \colon X \to Y$ from $X$ (the dataset, which we take to be a high-dimensional point cloud) to a lower-dimensional space $Y$, a \emph{cover} of the image $f(X)$ in $Y$, and a \emph{clustering} algorithm for the preimage of each cover element. Mapper is known to be sensitive to its parameters, and many researchers have studied how to optimize them \cite{johanssonCertified20, chalapathi2021adaptive, bui2020f,alvarado2023g, oulhaj2024differentiable, oulhaj2024deep}. 

In this paper, we illustrate Mapper's adaptability with respect to its parameters by providing constructions that show for a given dataset $X$ and graph $G$, there exists a lens function $f \colon X \to Y$ with a reference space $Y$ and a cover $\mathcal{U}$ of the image $f(X)$ in $Y$, that, together with a particular clustering algorithm, generate a Mapper graph isomorphic to $G$.

\section{Background}
\label{sec:background}

In this section, we collect the necessary terms related to Mapper graphs for the results of this work and recall the Mapper construction. We refer the reader to a standard text (e.g., \cite{MunkresTopology00}) for definitions and background in topology, and \cite{carlsson2021topological,matouvsek2003using} for an overview of applied topology and topological combinatorics, respectively. Throughout this paper, the notation $|\sigma|$ is used to denote the cardinality of a finite set $\sigma$.

\subsection{Simplicial Complexes}
We define simplicial complexes, which are often basic building blocks for extracting shapes from datasets. 

\begin{definition}[Simplicial Complex]
An \emph{(abstract) simplicial complex} $K$ is a collection of non-empty finite sets that is closed under the subset relation, i.e., if $\sigma \in K$ and $\tau\subseteq \sigma$, then $\tau \in K$. 
\end{definition}

An element $\sigma \in K$ is called a \emph{simplex} and the \emph{dimension} of $\sigma$ is $\dim(\sigma) \coloneq |\sigma|-1$. An $i$-dimensional simplex $\sigma$ is called an \emph{$i$-simplex}. In particular, $0$-simplices are \emph{vertices}, denoted by $V$, and $1$-simplices are \emph{edges}. The dimension of a simplicial complex $K$ is the maximum of the dimensions of simplices in $K$. We call a $1$-dimensional simplicial complex a \emph{graph}, which has no self-edges (also called loops) or multiple edges. It is referred to as a simple graph in the literature.

For a simplicial complex $K$, its \emph{geometric realization}, denoted by $\|K\|$, is a geometric simplicial complex whose faces correspond to those of $K$. Note that the topology of $\|K\|$ is a subspace topology of a Euclidean space and geometric realizations of $K$ are homeomorphic to each other. Let $e_0$ denote the zero vector and $e_i$ denote the $i$th standard basis vector for $i=1,2,\cdots,n-1$. With a bijection $\phi:V \to \{0,1, \cdots, n-1\}$ where $n:=|V|$, the simplicial complex $K$ can be realized in $\mathbb{R}^{n-1}$ as $\|K\|=\{c(\sigma) \mid \sigma \in K \}$, where $c(\sigma)$ is the convex hull of $\{e_{\phi(v)}\}_{v \in \sigma}$.

\subsection{Nerves and Nerve Theorems}
The Mapper construction involves a special type of simplicial complex, called a \emph{nerve}, which encodes the intersection patterns of a given collection of sets.  Nerve theorems under some assumptions show the relationship between a topological space and the topology of the nerve complex for a cover of the space. This provides the foundation of the Mapper construction.

\begin{definition}[Nerve]
Let $\mathcal{U}$ be a finite collection of sets. The \emph{nerve} of $\mathcal{U}$, denoted by $\mathrm{Nrv}(\mathcal{U})$, is the collection of subsets of $\mathcal{U}$ that consist of sets that have a non-empty intersection, i.e.,
\[
\mathrm{Nrv}(\mathcal{U})\coloneq\{\mathcal{U}' \subseteq \mathcal{U} \mid \bigcap_{U \in \mathcal{U}'}U \neq \emptyset\}.
\]
The nerve is also called the \emph{nerve complex} since a nerve is an abstract simplicial complex. See Figure \ref{fig:nerve} for an illustration of a nerve.
\end{definition}

 The basic nerve theorem of Leray \cite{leray1946l} states that if $\mathcal{U}$ is a finite collection of open sets in a topological space such that all intersections of sets in $\mathrm{Nrv}(\mathcal{U})$ are either contractible or empty (i.e., $\mathcal{U}$ is a \emph{good} open cover), then $\| \mathrm{Nrv}(\mathcal{U})\|$ is homotopy equivalent to $\bigcup_{U\in \mathcal{U}}U$. The homotopy equivalence is illustrated in Figure \ref{fig:nerve}.  More recent results show that the nerve theorem holds for closed convex covers in Euclidean spaces \cite{Bauer_2023}. 

\begin{figure}
    \centering
    \includegraphics[scale=0.3]{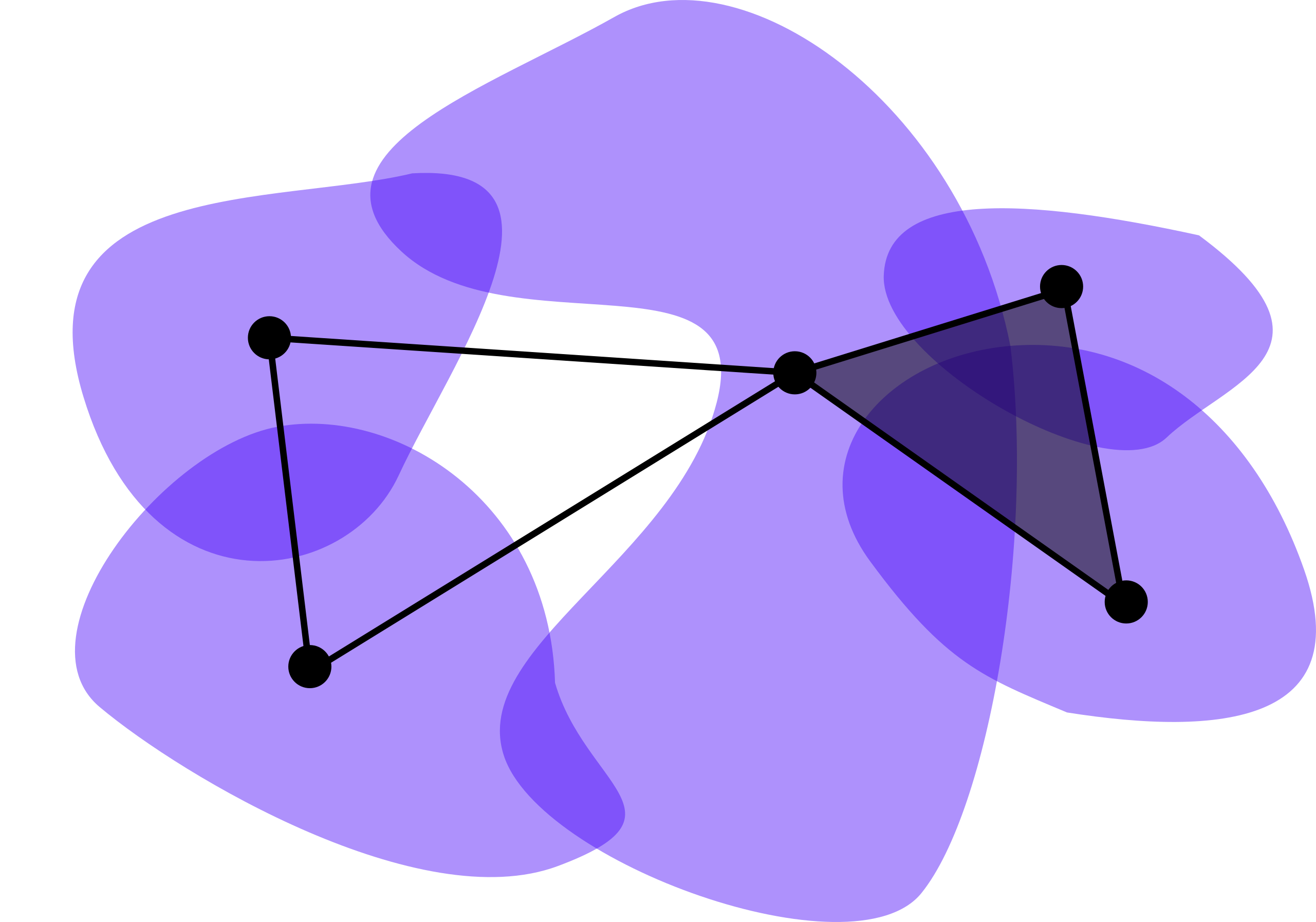}
    \caption{A collection $\mathcal{U}$ of sets and its nerve $\mathrm{Nrv}(\mathcal{U})$. The sets in $\mathcal{U}$ are colored in purple. Its nerve $\mathrm{Nrv}(\mathcal{U})$ consists of vertices, edges, and a triangle, and its geometric realization is colored in black. 
    The space $\bigcup_{U\in \mathcal{U}}U$ and the geometric simplicial complex $\|\mathrm{Nrv}(\mathcal{U})\|$ are homotopy equivalent, as shown in the nerve theorem.} 
    \label{fig:nerve}
\end{figure}

\subsection{Mapper Construction}
We now provide the \emph{Mapper} construction for a given set of parameter choices. In 2007, Singh, Mémoli, and Carlsson introduced the construction,  referring to it as the {\it statistical} version of Mapper \cite{singh2007topological}. It is also commonly known as the Mapper algorithm. For a dataset $X$, the Mapper construction produces a nerve with the following four steps: (See Figure \ref{fig:mapper-example} for an example of how to compute the Mapper graph for data points in $\R^2$.)
\begin{enumerate}
\item Select a \emph{lens} (or \emph{filter}) function $f\colon X \to Y$ from $X$ to a reference space $Y$;

\item Choose a cover $\mathcal{U}$ of the image $f(X)$ in $Y$, taken as an open or closed cover in practice \cite{carlsson2009topology};

\item Apply a clustering algorithm to $f^{-1}(U)$ for each $U \in \mathcal{U}$ and let $f^\ast(\mathcal{U})$ be the collection of resulting clusters, i.e., $f^\ast(\mathcal{U}) \coloneq \{\text{clusters of } f^{-1}(U) \mid U \in \mathcal{U}\}$;

\item Define the {\it Mapper construction} of $X$ to be the nerve $\mathrm{Nrv}(f^\ast(\mathcal{U}))$ of $f^\ast(\mathcal{U})$.
\end{enumerate}

The collection $f^\ast(\mathcal{U})$ in step 3 is regarded as a \emph{multiset}, i.e., identical clusters from different preimages $f^{-1}(U)$ are considered \emph{distinct} elements. In \cite{singh2007topological},
assigning a vertex to a cluster obtained from applying a clustering algorithm to each preimage $f^{-1}(U)$ illustrates that $f^\ast(\mathcal{U})$ is a multiset.

Since visualizing graphs is easier than visualizing high-dimensional simplicial complexes, in practice, step 4 often consists of only finding the $1$-skeleton of the nerve $\mathrm{Nrv}(f^\ast(\mathcal{U}))$. This means that the Mapper construction only focuses on singleton sets and pairs of sets with non-empty intersections. In the next section, we refer to this $1$-skeleton as the {\it Mapper graph} of $X$ and use the same notation for the graph as for the simplicial complex.

\begin{figure}[]
    \begin{center}
    \includegraphics[width=5 in]{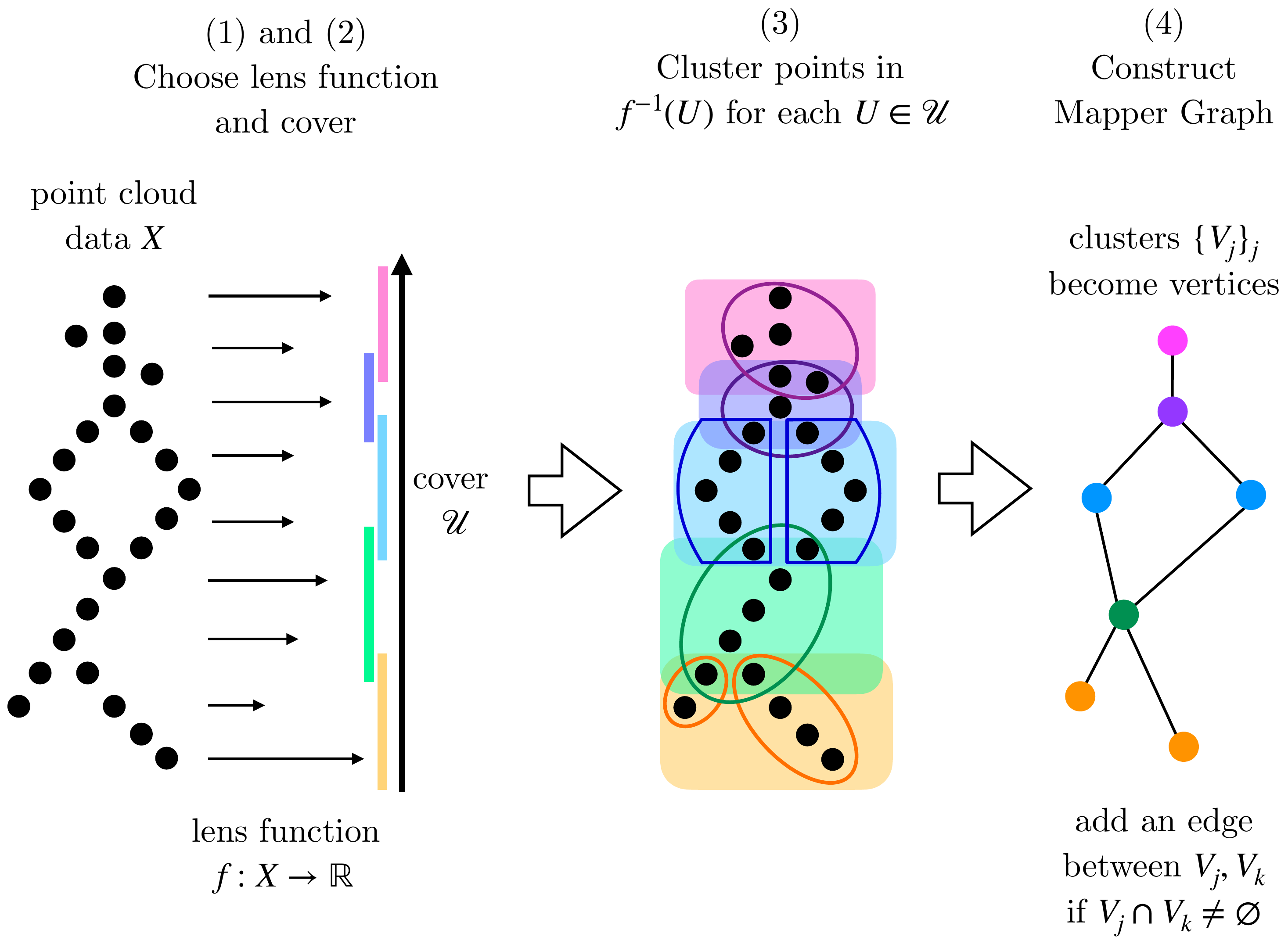}
    \caption{Mapper Graph Construction. An example of the four-step procedure to construct a Mapper graph from a finite set of points $X\subseteq \R^2$. In steps 1 and 2, the lens function $f\colon X\to \R$ and cover $\mathcal{U}$ of $f(X)$ are initialized by the user. The lens function is the height function and the cover is denoted by the colored overlapping intervals. In step 3, the preimage of each interval is denoted by the points in the shaded regions with the same color. A clustering algorithm chosen by the user is applied to each colored region. The points in each ellipsoid-shaped region are clustered together. In step 4, we construct the Mapper graph by creating a vertex for each cluster, and adding an edge between two vertices if the two clusters contain at least one point from $X$ in common. Observe the Mapper graph is the 1-skeleton of the nerve of the clustered sets in step 3.}
    \label{fig:mapper-example}
    \end{center}
\end{figure}

\section{Any graph is a Mapper graph}\label{trivial clustering}

In this section, we prove our main results showing that for a finite dataset $X$, any graph $G$ is a Mapper graph with two different constructions. One construction employs a star cover of a graph and the other uses convex subsets of $\mathbb{R}^3$. In both cases, if $X \subseteq \mathbb{R}^n$, we show that the lens function can be taken to be a continuous function. As a clustering algorithm, we choose the {\it trivial clustering algorithm}, which simply returns $\{A\}$ for any given set $A\subseteq X$. The Mapper construction with the trivial clustering algorithm can be formulated as follows.

\begin{definition}[Mapper with trivial clustering]\label{def: reebspace}
Let $f\colon X \to Y$ be a function from a set $X$ to a reference space $Y$. For any cover $\mathcal{U}$ of the image $f(X)$ in $Y$ and the cover $f^{-1}(\mathcal{U}) \coloneq \{f^{-1}(U) \mid U \in \mathcal{U}\}$ of $X$, 
we call the simplicial complex  
\[
\MT(\mathcal{U}, f) \coloneq \mathrm{Nrv}(f^{-1}(\mathcal{U}))
\]
the \emph{Mapper construction with trivial clustering}.
\end{definition}

\subsection{Using the star cover of a graph $G$}
\label{sec:graph-co-domain}

We use the given graph $G$ as the co-domain for the lens function $f$ with a \emph{star} cover of $G$. The \emph{star} of a vertex $v$, denoted by $\operatorname{st}(v)$, consists of the vertex $v$ along with all edges containing $v$. The \emph{star cover} is the collection of all stars of vertices of $G$. We equip the graph $G$ with the \emph{Alexandrov topology}, where the open sets are unions of stars \cite{alexandroff37}.
With this topology, the star cover forms an open cover of $G$. We then select a lens function that maps at least one point of a dataset $X$ to each edge or isolated vertex of $G$. The Mapper construction with these parameters leads to Theorem~\ref{theorem:trivial-mapper-star}. See Figure \ref{fig:star-open-cover} for an example of how to construct the cover and lens function. 

\begin{figure}[htp]
    \begin{center}
    \includegraphics[width=5.5 in]{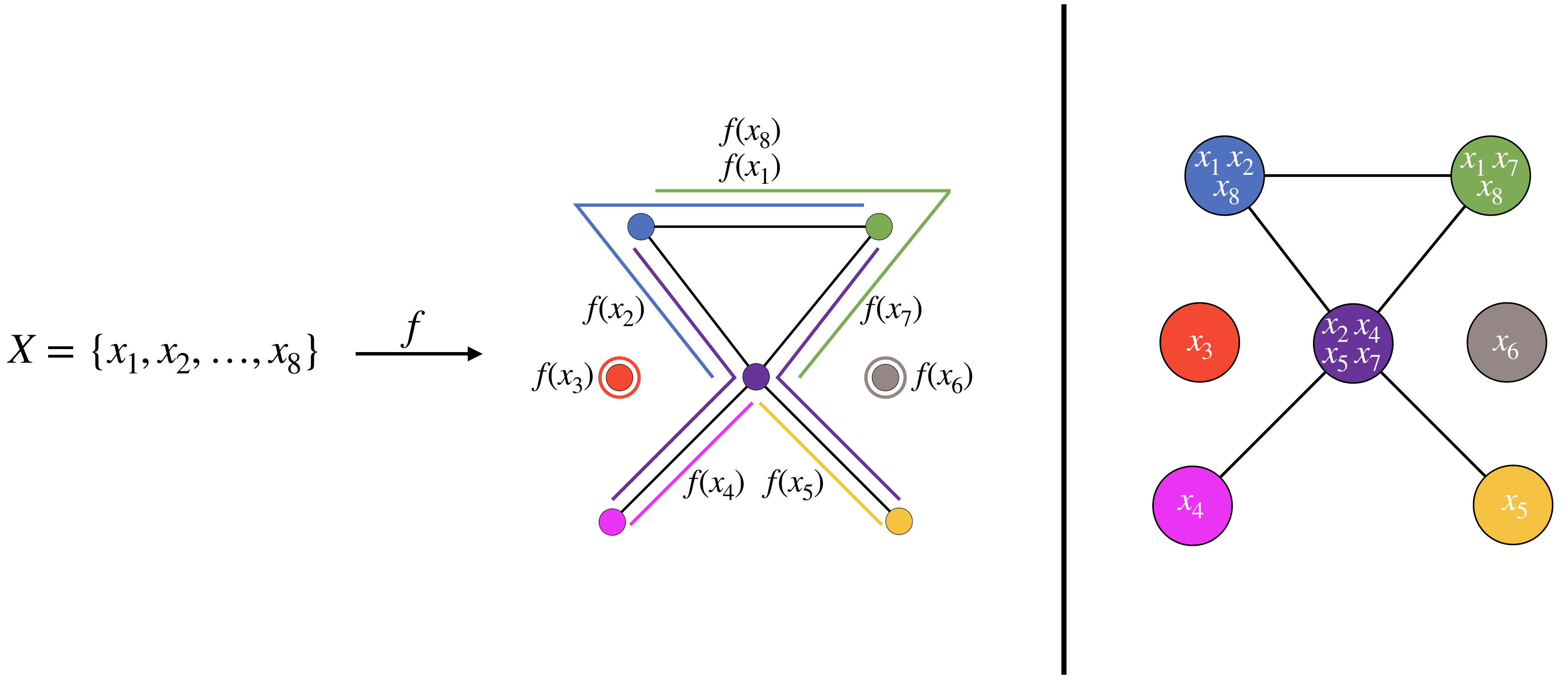}
    \caption{Mapper graph with star cover reconstruction. Left: Each color represents the star of a vertex. The set of all stars of vertices forms the cover. The lens function $f$ maps at least one element of $X$ to each edge and isolated vertex in $G$. Right: The Mapper graph is constructed using the function $f$ and the star cover. We obtain seven vertices that are colored according to their corresponding cover element. The elements of $X$ listed on each vertex are those mapped to that cover element. The Mapper graph is isomorphic to the original graph $G$. }
    \label{fig:star-open-cover}
    \end{center}
\end{figure}

\begin{theorem} %
\label{theorem:trivial-mapper-star}
Let $X$ be a set of points. If $G$ is a graph with vertex set $V$, edge set $E$, and isolated vertex set $I$ such that $|X| \geq |E|+|I|$, then there exists a cover $\mathcal{U}$ of $G$, and a function $f\colon X\to G$
such that the Mapper graph $\MT(\mathcal{U}, f)$ with trivial clustering is isomorphic to $G$.
\end{theorem}

\begin{proof}
Let $\mathcal{U}$ be the {\it star cover} of $G$ defined as $\mathcal{U}\coloneq\{\operatorname{st}(v) \mid v \in V\}$. Since $|X| \geq |E|+|I|$, we can define a function $f \colon X \to G$ such that $f$ is surjective onto the edges and isolated vertices of $G$ (i.e., each edge or isolated vertex has at least one point that is mapped onto it). Note that the vertex set of the Mapper graph with trivial clustering $\MT(\mathcal{U}, f)$ is given by $f^{-1}(\mathcal{U})$. Consider a bijection $\varphi \colon V \to f^{-1}(\mathcal{U})$ defined by $\varphi(v)=f^{-1}(\operatorname{st}(v))$ for each $v \in V$. Take $v_1, v_2 \in V$. By the definition of the nerve,
 $\{\varphi(v_1),\varphi(v_2)\}$ is an edge of the Mapper graph if and only if $f^{-1}(\operatorname{st}(v_1)) \cap f^{-1}(\operatorname{st}(v_2))$ is non-empty.
 Since $f$ is surjective onto the edges and isolated vertices of $G$, we have $f^{-1}(\operatorname{st}(v_1)) \cap f^{-1}(\operatorname{st}(v_2)) \neq \emptyset$ if and only if  $\operatorname{st}(v_1) \cap \operatorname{st}(v_2) \neq \emptyset$. This occurs if  and only if $\{v_1,v_2\} \in E$. Therefore, $\{\varphi(v_1),\varphi(v_2)\}$ is an edge of $\MT(\mathcal{U}, f)$ if and only if
 $\{v_1,v_2\} \in E$, which shows that $\MT(\mathcal{U}, f)$ is isomorphic to $G$. 
\end{proof}

Theorem~\ref{theorem:trivial-mapper-star} is admittedly quite contrived since using a graph $G$ as the co-domain of the lens function is unlikely to be anyone's first choice for a Mapper parameter. However, we are still able to construct a lens function and cover of $\R^3$ such that any graph is a Mapper graph. Any graph can be realized in $\mathbb{R}^d$ for $d\geq 3$, and the star of a vertex in $G$ naturally corresponds to the star of the vertex in the realization. We modify the lens function in Theorem~\ref{theorem:trivial-mapper-star} so that it maps points to \emph{points} in edges instead of edges themselves, and \emph{points} for isolated vertices instead of isolated vertices themselves. This construction generates the Mapper graph isomorphic to $G$. The proof is analogous to that of Theorem~\ref{theorem:trivial-mapper-star}.

\subsection{Using convex subsets of $\R^3$}
\label{sec:convex-subsets} \emph{Convex} subsets of $\R^3$ instead of star-shaped subsets would be an attractive choice for a cover in the Mapper construction. We provide a convex cover of $\R^3$ and a lens function such that any graph can be represented as a Mapper graph. For this construction, we employ a result on \emph{$d$-representable complexes}, which are simplicial complexes $K$ that are isomorphic to the nerve of a finite collection of convex sets in $\R^d$. Wegner presented the representability and its optimality for graphs as $1$-dimensional simplicial complexes in 1967 \cite{wegner1967eigenschaften} and Perel'man rediscovered this result in 1985 \cite{perelman85realization}. Refer to Section~\ref{sec:generaliziation} for results concerning
simplicial complexes.
\begin{theorem}[\cite{wegner1967eigenschaften}] A graph $G$ is $3$-representable, with $3$ being the smallest possible value.
\label{theorem:representability}
\end{theorem}

A proof sketch of Theorem~\ref{theorem:representability} is provided in \cite[Theorem 3.1]{tancer13intersection}. Theorem~\ref{theorem:representability} asserts the existence of a finite collection of convex subsets in $\R^3$ whose nerve is isomorphic to a given graph $G$. 
This collection will cover the image of a lens function with a co-domain in $\mathbb{R}^3$. The function maps our data points to the pairwise intersections between the convex sets, and to the convex sets corresponding to isolated vertices. This construction generates the Mapper graph isomorphic to a given graph $G$. See Figure \ref{fig:convex-open-cover} for an illustration of the construction.
\vskip 10pt
\begin{figure}[h]
    \begin{center}
    \includegraphics[width=5.5 in]{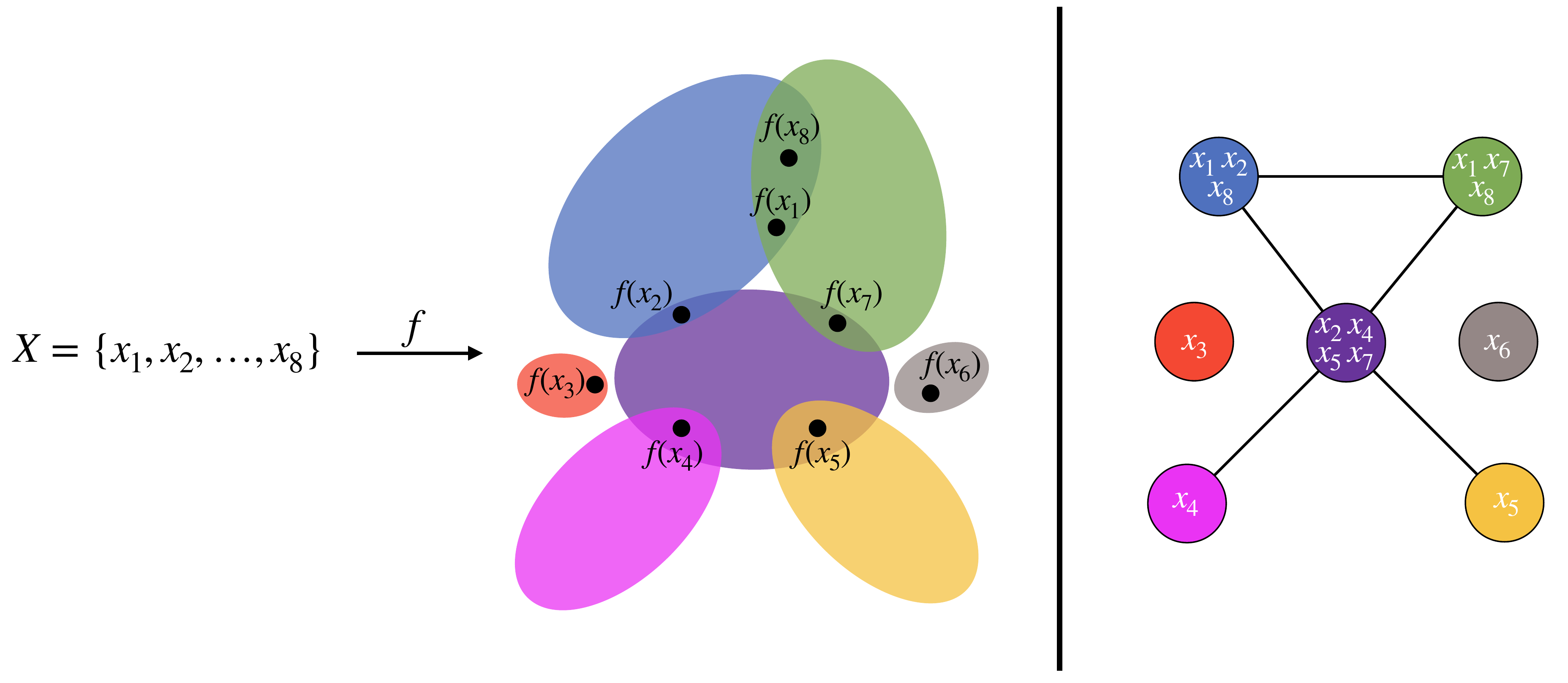}
    \caption{Mapper graph with convex sets reconstruction. We consider $X=\{x_1, x_2, \dots, x_8\}$ and the same graph as in Figure~\ref{fig:star-open-cover}. Left: Each color represents a convex set. The set of all convex sets forms a cover. The lens function $f$ maps at least one element of $X$ to each intersection of convex sets and isolated convex set. Right: The Mapper graph is constructed using the function $f$, the convex cover, and trivial clustering. We obtain seven vertices that are colored according to their corresponding cover element. The elements of $X$ listed on each vertex are those mapped to that cover element. The Mapper graph is isomorphic to the original graph.}
    \label{fig:convex-open-cover}
    \end{center}
\end{figure}
\vskip 5pt
\begin{theorem}\label{thm: Any graph is a mapper graph v2}
Let $X$ be a set of points. If $G$ is a graph with edge set $E$ and isolated vertex set $I$ such that $|X| \geq |E|+|I|$, then there exists a collection $\mathcal{C}$ of convex sets in $\mathbb{R}^3$ and a function \mbox{$f \colon X \to \mathbb{R}^3$} such that the Mapper graph $\MT(\mathcal{C}, f)$ with trivial clustering is isomorphic to $G$. 
\label{theorem:convex-sets}
\end{theorem}

\begin{proof}
Given our graph $G$, Theorem~\ref{theorem:representability} guarantees that there exists a collection $\mathcal{C}$ of convex subsets of $\mathbb{R}^3$ such that $\mathrm{Nrv}(\mathcal{C})$ is isomorphic to $G$. Hence, by the definition of the nerve, $|\{ \{C_1,C_2\} \subseteq \mathcal{C}\mid C_1 \cap C_2 \neq \emptyset, C_1 \ne C_2\}|=|E|$. Since $|X| \geq |E|+|I|$, we can define a lens function $f \colon X \to \bigcup_{C\in \mathcal{C}} C$ so that non-empty intersections between pairs of the convex sets and isolated convex sets have images of $f$, i.e., 

(1) $f(X) \cap (C_1 \cap C_2) \neq \emptyset$ for any pair $C_1, C_2 \in \mathcal{C}$ with non-empty intersection, and 

(2) $f(X) \cap C \neq \emptyset$ for $C \in \mathcal{C}$ with $C \cap\left(\bigcup{ \{C' \in \mathcal{C} \mid C' \ne C \}}\right) =\emptyset$. 

Therefore, with such a cover $\mathcal{C}$ of $f(X)$, we see that $\MT(\mathcal{C}, f)$ is isomorphic to $G$.
\end{proof}

Theorem~\ref{theorem:convex-sets} states that any graph is a Mapper graph with trivial clustering even using a cover to be made up of convex subsets of $\R^3$. This construction could potentially be implemented. The proof of Theorem~\ref{theorem:representability} is constructive and algorithmic concerning polytopes, Schlegal diagrams, and projections. However, implementations that compute these convex sets have not been developed to the best of our knowledge. If these convex sets are computable, then one could implement an algorithm demonstrating any graph is a Mapper graph with trivial clustering.

\subsection{Inverse Problem with Continuous Lens Functions}\label{sec: continuous extensions}

While the lens function in Section~\ref{sec:convex-subsets} that maps points to intersections of convex sets is still somewhat contrived, we note that it can be formulated as a continuous function. We show that the lens function $f:X \to Y$ can be viewed as a continuous function and that $f$ can be extended when $X$ is a closed subset of $\R^n$, as in the case of finite point cloud datasets.

If the dataset $X$ does not come from any natural underlying topological space, we can generate a topology on $X$ using the preimages of the open sets in $Y$. %
If our dataset $X$ is contained in a ``nice'' topological space, such as $\mathbb{R}^n$, we can extend $f$ to a continuous function $F$ on the whole space.
This follows from the well-known Tietze extension theorem (see, for instance, \cite{MunkresTopology00}), which implies that any $\R^d$-valued continuous function on a closed subset of a {\it normal} topological space can be extended to a continuous function on the whole space.

Furthermore, when our dataset $X$ is contained in $\R^n$, we have Lipschitz extensions and $m$-times continuously differentiable (denoted as $C^m$) extensions. Recall that a function $g:Z \to \R^d $ is \emph{Lipschitz continuous} if there exists a real constant $L\geq 0$ such that $|f(z_1) - f(z_2)| \leq L|z_1 - z_2|$ for all $z_1, z_2 \in Z$, and the \emph{Lipschitz constant} $\mathrm{Lip}(g)$ of $g$ is the smallest (infimum) of all such constants $L$.

If one does not care about increasing the Lipschitz constant $\mathrm{Lip}(f) > 0$ of $f = (f_1, \dots, f_d) \colon X\to \R^d$ by a constant factor, we have an explicit formula for its Lipschitz extension.
The function $F = (F_1, \dots, F_d) \colon \R^n \to \R^d$ defined by 
\[
F_i(y) := \inf_{x \in X} \{f_i(x) + \mathrm{Lip}(f)|x - y|\}\quad \text{ for each } i = 1, \dots, d
\]
is a Lipschitz extension of $f$, and the Lipschitz constant of $F$ is $\sqrt{d} \mathrm{Lip}(f)$~\cite[Theorem 3.1]{evans2018measure}. 
On the other hand, if we want to preserve the Lipschitz constant of $f$, Kirszbraun's theorem states that an $\R^d$-valued Lipschitz function defined on an arbitrary subset of $\R^n$ can be extended to a Lipschitz function on all of $\R^n$ with the {\it same} Lipschitz constant (see~\cite[2.10.43]{federer2014geometric} for an existence proof). 
As for $C^m$ extensions, Whitney's extension theorem (see~\cite[Theorem 6.10]{evans2018measure} for a proof of the case when $m = 1$) guarantees the existence of an $m$-times continuously differentiable extension of $f$. 
Seeing how ``nice'' (i.e., small $C^m(\R^n)$ norm) the $C^m$ extension can be is an active area of research~\cite{fefferman2009whitney, fefferman2020fitting}.

In the context of Section \ref{sec:convex-subsets}, recall that we construct a lens function $f$ that maps our dataset $X$ into a union $\Omega$ of convex sets in $\R^3$. The image of the extension $F \colon \R^n \to \R^3$ of $f \colon X \to \R^3$ obtained with the Tietze, Kirszbraun, or Whitney extension theorem is not guaranteed to be contained in $\Omega$. 
However, if all convex sets are 3-dimensional and $f$ maps into the interior of $\Omega$, we can use Kirszbraun's extension theorem (for example) to ensure that $F$ maps small enough open balls around each data point in $X$ into $\Omega$.

\begin{prop}\label{prop lipschitz extensions}
Let $X = \{x_i\}$ be a finite subset of $\R^n$, and let $\mathcal{C}$ and $f \colon X \to \mathbb{R}^3$ be as in Theorem~\ref{thm: Any graph is a mapper graph v2}. Suppose that each convex set in $\mathcal{C}$ has a non-empty interior and that $f(X)$ is contained in the interior of $\Omega:= \cup_{C \in \mathcal{C}} C$. Then there exists a Lipschitz extension $F: U \to \mathbb{R}^3$ of $f$ to the union $U := \cup_i B(x_i, r_i)$ of open balls of radius $r_i > 0$ centered at each point $x_i$ in $X$ such that $\mathrm{Lip}(F) = \mathrm{Lip}(f)$. %
\end{prop}

\begin{proof}
Since $X$ is finite, $\mathrm{Lip}(f) < \infty$.
For simplicity, denote $L:=\mathrm{Lip}(f)$. By Kirszbraun's theorem, there exists an extension $F \colon \R^n \to \R^3$ of $f$ with $\mathrm{Lip}(F) = L$, which means that $F(x) = f(x)$ for all $x\in X$ and
\begin{equation}\label{eq: lipschitz definition}
|F(x) - F(y)| \leq L|x - y| \quad \text{ for all } x, y \in \R^n.
\end{equation}
Fix $i$ and let
\[
\delta_i := \inf_{y \in \R^n\setminus \Omega} |F(x_i) - y|.
\]
Since $f(X)$ is contained in the interior of $\Omega$ by assumption, the same holds for $F(X)$, which shows $\delta_i>0$. Thus, $B(F(x_i), \delta_i) \subseteq \Omega$. 
To prove $F(B(x_i, \delta_i/L)) \subseteq B(F(x_i), \delta_i)$, note that for $x \in B(x_i, \delta_i/L)$,
\begin{align*}
    |x_i-x|<\frac{\delta_i}{L}, \mbox{ or }
    L|x_i-x|<\delta_i.
\end{align*}
Then by (\ref{eq: lipschitz definition}),
\[
|F(x_i)-F(x)|<L\, |x_i-x|<\delta_i,
\]
which shows $F(x) \in B(F(x_i), \delta_i)$. Consequently, $F(B(x_i, \delta_i/L)) \subseteq B(F(x_i), \delta_i) \subseteq \Omega$ and
therefore $F(U)\subseteq \Omega$ for $U = \cup_i B(x_i, \delta_i/L)$.
\end{proof}
We can now make the following observation: 
If one wants to update the dataset to include new data points $X'$ which lie within $U$, the Mapper graph with trivial clustering of $F \colon X \cup X' \to \Omega$ will be equal to that of $f \colon X \to \Omega$ for the same collection of convex sets.
For this purpose, one may want to optimize the initial Lipschitz function $f \colon X \to \Omega$ so that $U$ is as large as possible. %

\section{Generalizing Mapper Constructions to Simplicial Complexes}
\label{sec:generaliziation}

In this section, we will briefly see how our results from Section \ref{trivial clustering} generalize to general simplicial complexes.
We will see how our constructions in Sections \ref{sec:graph-co-domain}, and \ref{sec:convex-subsets} generalize to show that for any dataset $X$, any {\it small enough} simplicial complex $K$ is a Mapper construction (Mapper simplicial complex). 
By ``small enough'' simplicial complex we mean that $|X| \geq |K\setminus V|+|I|$, where $V$ is the vertex set of $K$, $K\setminus V$ is the set of all $i$-faces of $K$ with $i \geq 1$, and $I$ is the set of isolated vertices of $K$.
This is analogous to the condition that $|X| \geq |E|+|I|$ where $E$ is the edge set and $I$ is the set of isolated vertices from Theorems~\ref{theorem:trivial-mapper-star} and \ref{theorem:convex-sets}.
We will also discuss how the continuous (e.g. Lipschitz continuous) extensions from Section \ref{sec: continuous extensions} can be applied to the generalization of the results in Section \ref{sec:convex-subsets}.

We begin by generalizing Theorem~\ref{theorem:trivial-mapper-star} to the case of simplicial complexes $K$ endowed with the Alexandrov topology. 
We consider the (simplicial) star open cover of $K$ along with a lens function $f\colon X \to K$ that maps at least one point $X$ to each isolated vertex, and each $i$-face of $K$ for $i\geq 1$ instead of each edge of a graph.

\begin{theorem} 
\label{theorem:trivial-mapper-construction-star}
Let $X$ be a set of points. 
If $K$ is a simplicial complex with vertex set $V$ and isolated vertex set $I$ such that $|X| \geq |K\setminus V|+|I|$, then there exists a cover $\mathcal{U}$ of $K$, and a function $f\colon X\to K$ such that the Mapper construction with trivial clustering $\MT(\mathcal{U}, f)$ is isomorphic to $K$.
\end{theorem}

Since any simplicial complex of dimension $d$ can be realized in $\mathbb{R}^{2d+1}$ where $2d+1$ is as small as possible (by the van Kampen-Flores theorem) \cite{van1933komplexe,flores1933n},
we have a star cover in $\mathbb{R}^{2d+1}$ along with a lens function that maps at least one point to each intersection of stars and points for isolated vertices. Similarly, Theorem~\ref{theorem:convex-sets} also holds in $\mathbb{R}^{2d+1}$ due to the generality of  Theorem~\ref{theorem:representability} on $d$-representable complexes \cite{wegner1967eigenschaften,tancer2011representability}, using a lens function that maps at least one point to  intersections of convex sets and convex sets for isolated vertices.

\begin{theorem}\label{thm: convex sets in R2d+1}
Let $X$ be a set of points. If $K$ is a simplicial complex with vertex set $V$ and isolated vertex set $I$ such that $|X| \geq |K\setminus V|+|I|$, then there exists a collection $\mathcal{C}$ of convex sets in $\mathbb{R}^{2d+1}$ and a function \mbox{$f \colon X \to \mathbb{R}^{2d+1}$} such that the Mapper construction $\MT(\mathcal{C}, f)$ with trivial clustering is isomorphic to $K$. 
\label{theorem:convex-sets-d}
\end{theorem}

Finally, we note that the Lipschitz extension used in Proposition \ref{prop lipschitz extensions} is general enough to apply to the function $f \colon X \to \mathbb{R}^{2d + 1}$ constructed in Theorem \ref{thm: convex sets in R2d+1} as follows. 
\begin{prop}\label{prop: extension in general Euclidean space}
Let $X = \{x_i\}$ be a finite subset of $\R^n$, and let $\mathcal{C}$ and $f \colon X \to \mathbb{R}^{2d + 1}$ be as in Theorem~\ref{thm: convex sets in R2d+1}. Suppose that each convex set in $\mathcal{C}$ has a non-empty interior and that $f(X)$ is contained in the interior of $\Omega:= \cup_{C \in \mathcal{C}} C$. 
Then there exists a Lipschitz extension $F: U \to \mathbb{R}^{2d+1}$ of $f$ to the union $U := \cup_i B(x_i, r_i)$ of open balls of radius $r_i > 0$ centered at each point $x_i$ in $X$ such that $\mathrm{Lip}(F) = \mathrm{Lip}(f)$.
\end{prop}

\section{Conclusion}
In this paper, we showed that, given any point cloud and any simplicial complex, there exist parameters such that the Mapper construction of the input point cloud results in the given simplicial complex under a mild condition. Additionally, this result can still be achieved when limiting the Mapper algorithm to parameters that arise more naturally in practice, such as requiring the lens function to be a $C^n$ or Lipschitz continuous function, and requiring the open cover to consist of convex sets. Given this flexibility of the Mapper algorithm, we encourage users to carefully consider the effects of parameter choices on the resulting Mapper construction.

\noindent \textbf{Acknowledgements.} This research is a product of one of the working groups at the American Mathematical Society (AMS) Mathematical Research Community: \textit{Models and Methods for Sparse (Hyper)Network Science} in June 2022. The workshop and follow-up collaboration was supported by the National Science Foundation under grant number DMS 1916439. Any opinions, findings, and conclusions or recommendations expressed in this
material are those of the author(s) and do not necessarily reflect the views of the National Science Foundation or the American Mathematical Society.

The authors thank Greg Henselman-Petrusek for his careful reading and valuable comments.
\bibliographystyle{plain}
\bibliography{bibliography}

\end{document}